\documentclass[a4paper,12pt,english]{smfart}
\usepackage{amsmath}
\usepackage{amssymb}
\usepackage{amsthm}
\usepackage{amscd}
\usepackage{amsfonts}
\usepackage{graphicx}%
\usepackage{fancyhdr}
\usepackage{mathdots}
\usepackage[francais,english]{babel}
\usepackage{smfthm}

\usepackage[all]{xy}
\theoremstyle{plain} \numberwithin{equation}{section}
\newcommand{\0}{\mathcal{O}}
\newcommand{\D}{\mathrm{diag}}

\newcommand{\tG}{\widetilde{\mathrm{GL}}}
\newcommand{\G}{\mathrm{GL}}

\title{Transfer factors for Jacquet-Mao's metaplectic fundamental lemma}
\alttitle{Facteur de transfert pour le lemme fondamentale métaplectique de Jacquet-Mao}
\author{Viet Cuong Do}
\address{Faculty of Mathematics-Mechanics-Informatics\\
VNU University of Science\\334 Nguyen Trai, Thanh Xuan, Ha Noi, Vietnam. }
\email{vcuong.do@hus.edu.vn}
\date{}
\begin{document}
\frontmatter
\begin{abstract}
	In an earlier paper we proved Jacquet-Mao's metaplectic fundamental lemma which is the identity between two orbital integrals (one is defined on the space  of symmetric matrices and another one is defined on the $2$-fold cover of the general linear group) corrected by a transfer factor. In this paper, we restricted our calculation to the case where the relevant representative is a diagonal matrix. The purpose of the present paper is to show that we can extend this result for the more general relevant representative. Our proof is based on the concept of Shalika germs for certain Kloosterman integrals. 
\end{abstract}
\begin{altabstract}
	Dans un article pr\'ec\'edent, nous avons prouv\'e le lemme fondametal m\'etaplectique de Jacquet-Mao qui est l'indentit\'e entre deux int\'egrales orbitales (l'une est d\'efinie sur l'espace des matrices sym\'etriques et l'autre est d\'efine sur le rev\^etement \`a deux feuillets du groupe g\'en\'eral lin\'eaire) corrig\'ee par un facteur de transfert. Dans cet article, nous avons limit\'e notre calcul au cas o\`u le repr\'esentant pertinent est une matrice diagonale. Le but du pr\'esent article  est de montrer que nous pouvons \'etendre ce r\'esultat au repr\'esentant pertinent plus g\'en\'eral. Notre preuve est bas\'ee sur le concept des germes de Shalika pour certains int\'egraux Kloosterman.
\end{altabstract}
\thanks{This research is supported by the Vietnam National University, Hanoi (VNU) under project number QG.19.06. The first version of this paper is written during my visit at Max Planck Institute for Mathematics (MPIM). This manuscript is written during my visit at Vietnam Institute for Advanced Study in Mathematics (VIASM). The author would like to thank the MPIM and the VIASM for a very pleasant and productive visit.}
\altkeywords{Germe de Shalika, facteur de transfert, lemme fondamental, correspondance métaplectique}
\maketitle
\mainmatter
\section{Introduction} Let $K$ be a global field (that is: a number field or the function field of a curve over a finite field) and $\mathbb{A}$ be its ring of ad\`eles. Jacquet conjectured that the cuspidal representation of $\G_r(\mathbb{A})$ distinguished by a general orthogonal subgroup should be the lifting of a cuspidal representation of its metaplectic cover $\tG_r(\mathbb{A})$ (a certain twofold cover of $\G_r$). Jacquet and Mao have suggested that to solve this conjecture, we establish two relative trace formulas (one for the group $\G_r$ and one for its metaplectic cover $\tG_r$) and then compare them. Roughly speaking, the relative trace formula attached to a group is an identity between two expansion of the certain integral, know as ``\textit{geometric expansion}'' and ``\textit{spectral expansion}''. The terms of geometric expansion are quite explicit  but complicated. The terms of spectral expansion contain information of automorphic representations. By comparing the geometric side of two relative trace formulas, we obtain then a comparison between the two spectral sides.   

One of the step of this approach is precisely the fundamental lemma which we state now.


Let $F$ be a non-Archimedean local field, $\0$ be its valuation ring, and $k$ be its residue field. Assume that the cardinality $q$ of $k$ is odd. 
We choose once for all an uniformizer $\varpi$ of $\0$ (i.e a generator of the maximal ideal of $\0$). We write $v$ for the valuation of $F$ and $|.|$ for the norm, normalised such that $|x|=q^{-v(x)}$.

Let $B_r$ be the standard Borel subgroup of $\G_r$ (the subgroup of invertible upper triangular matrices) with unipotent radical $N_r$, and let $T_r$ be the maximal split torus contained in $B_r$. Let $S_r$ be the variety $\{g\in\G_r|{}^tg=g\}$ and $W_r$ be the Weyl group of $T_r$. 

Let $\psi:F\to \mathbb{C}^*$ be a non-trivial additive character of level 0. We define  then a character $\theta:N_r(F)\to \mathbb{C}^*$ of $N_r(F)$ : $\theta(n)=\psi\left(\frac{1}{2}\sum_{i=2}^rn_{i-1,i}\right)$.

The local metaplectic cover $\tG_r(F)$ of $\G_r(F)$ is an extension of $\G(F)$ by $\{\pm 1\}$ (cf. \cite{KP}). We can write the elements of $\tG_r(F)$ in the form $\widetilde{g}=(g,z)$, with $g\in \G_r(F)$ and $z\in \{\pm 1\}$ and the group multiplication is defined by
$$(g,z)(g,z')=(gg',\chi(g,g')zz'),$$
where $\chi$ is a certain cocycle (cf. loc. cit. for the definition of $\chi$). This cover splits (canonically) over $N_r(F)$ (the splitting $\sigma$ over $N_r(F)$ is simply defined by $\sigma(n)=(n,1)$); it splits also over $\G_r(\0)$. The splitting $\kappa^*$ over $\G_r(\0)$ is defined by $\kappa^*(g)=(g,\kappa(g))$. We denote by $\G_r^*(\0)$ the image of $\G_r(\0)$ via the splitting $\kappa^*$.

We say a function $f$ on $\tG_r(F)$ is \textit{genuine} if it satisfies $f(g,z)=f(g,1).z$.

Let $\mathcal{H}_r$ be the set of the smooth functions with compact support on $\G_r(F)$ who are bi-$\G_r(\0)$-invariant. This $F$-vector space is equipped an algebraic structure by the convolution
$$\phi*\phi'(x)=\int_{\G_r(F)}\phi(g)\phi'(g^{-1}x)dg.$$
Its unit element is the function defined by
$$\phi_0(g)=\begin{cases}1 &\text{if }g\in \G_r(\0)\\
0&\text{otherwise.}\end{cases}$$

Let $\widetilde{\mathcal{H}}_r$ be the set of the smooth functions with compact support on $\tG_r(F)$ who are bi-$\G_r^*(\0)$-invariant and who are genuine. This $F$-vector space is equipped an algebraic structure by the convolution
$$f*f'(\widetilde{x})=\int_{\G_r(F)}f((g,1))f'((g,1)^{-1}\widetilde{x})dg.$$
Note that the function $\widetilde{g}\mapsto f(\widetilde{g})f'(\widetilde{g}^{-1}\widetilde{x})$ on $\tG_r(F)$ is $\{\pm 1\}$-invariant and that the integration is over $\G_r(F)\simeq \{\pm 1\}\setminus\tG_r(F)$. Its unit element is the function defined by
$$f_0((g,1))=\begin{cases}\kappa(g) &\text{if }g\in \G_r(\0)\\
0&\text{otherwise.}\end{cases}$$

The group $N_r$ acts on $S_r$ by $n.s={}^tnsn$ and $N_r\times N_r$ acts on $\G_r$ by $(n,n').g=n^{-1}gn$. We say an orbit $N_rs$ (resp. $(N_r\times N_r)g$) is \textit{relevant} if the restriction of $\theta^2$ (resp. $(n,n')\mapsto \theta(n^{-1}n'$)) on the stabilizer $(N_r)_{s}$ (resp. $(N_r\times N_r)_{g}$) of $s$ (resp. of $g$) is trivial. 

Consider the standard Levi-subgroup $M$ of $\G_r$ of type $(r_1,\dots,r_m)$. Thus $M$ is the group of matrices of the form $\D(g_i)$ with $g_i\in \G_{r_i}$. We denote by $w_{\G_r}$ the longest Weyl element of $\G_r$ (i.e the $r\times r$ permutation matrix whose entries are one on the second diagonal and whose other entries are 0). Let $w_M=\D(w_{\G_{r_i}})$. Let $T_M$ be the group of matrices of the form $\D(a_i\mathrm{Id}_{r_i})$ where $\mathrm{Id}_{r_i}$ is the identity matrix of size $r_i$ and $a_i\in F^*$ - the center of $M$. It follows from \cite[Theorem 1]{M} that the elements of the form $w_M\mathbf{t}$ (resp. of the form $w_{\G_r}w_M\mathbf{t}$) with $\mathbf{t}\in T_M$ (when $M$ runs through the set of standard Levi-subgroup of $\G_r$) form a system of representatives for the relevant orbits of $N_r$ (resp. of $N_r\times N_r$). So we have then a bijection between the sets of relevant orbits : $w_M\mathbf{t}\mapsto w_{\G_r}w_M\mathbf{t}$. We denote by $W_r^R$ the set of relevant elements in $W_r$. If $w\in W_r^R$ then the unique $M$ such that $w=w_M$ is denoted by $M_{w}$. We also write $T_w$ for $T_M$. For instance, if $w=\mathrm{Id}_r$ then $M_{\mathrm{Id}_r}=T_{\mathrm{Id}_r}=T_r$ and if $w=w_{\G_r}$ then $M_{w_{\G_r}} =\G_r$ and $T_{w_{\G_r}}(F)=T_{\G_r}=\{\beta\mathrm{Id}_r|\beta\in F^*\}\simeq F^*$. The stabilizer of $\mathbf{t}\in T_r(F)$ (resp. $w_{\G_r}\mathbf{t}$) in $N_r$ (resp. $N_r\times N_r$) is trivial. In this sense the diagonal matrices are representatives of the largest orbits. From now on, we shall drop the subscript $M$ in the notation $w_M\mathbf{t}$ when we don't want to specify what is the standard Levi subgroup. 

We denote by $\mathcal{C}_c^\infty(S_r(F))$ (resp. $\mathcal{C}_c^\infty(\tG_r(F))$) the space of the smooth function of compact support on $S_r(F)$ (resp. on $\tG_r(F)$). Let $\phi$ be a function in $\mathcal{C}_c^\infty(S_r(F))$ and $f$ be a genuine function in $\mathcal{C}_c^\infty(\tG_r(F))$. For each $w\textbf{t}$ as above, we consider the orbital integrals of the form :

$$I(w{\bf t},\phi)=\int_{N_r/(N_r)_{w{\bf t}}}\phi({}^tnw{\bf t}n)\theta^2(n)dn$$
and
\begin{align*}
J(w{\bf t},f)=\int_{N_r\times N_r/(N_r\times N_r)_{w_{\G_r}w{\bf t}}}f(\sigma(n)^{-1}(w_{\G_r}w{\bf t},1)\sigma(n'))\theta(n^{-1}n')dndn'.
\end{align*} 

If $\phi\in \mathcal{H}_r$, then the function $\phi_{|S_r(F)}\in \mathcal{C}_c^\infty(S_r(F))$. By abusing the notation $\mathcal{H}_r$ for the algebra $\{\phi_{|S_r(F)}|\phi\in\mathcal{H}_r\}$, our fundamental lemma is the following conjecture.
\begin{conj}[Jacquet-Mao] There exists a homomorphism  $h: \widetilde{\mathcal{H}}_r\to \mathcal{H}_r$ such that
	$J(w{\bf t},f)=\Delta(w{\bf t})I(w{\bf t},h(f))$, where $f\in \widetilde{\mathcal{H}}_r$ and $\Delta(w{\bf t})$ is an explicit transfer factor. 
\end{conj}

Since $h$ is an homomorphism between two algebras, it should send the unit element of the one to the unit element of the other. The (suggested) transfer factor for the largest orbit ($w=\mathrm{Id}_r$) is then calculated by the following propositions. 

\begin{prop}[cf. \cite{VC},\cite{VC1}]\label{result} Let $F$ be a local field of positive characteristic. Let $\textbf{t}=\D(t_1,\dots,t_r)$, we denote by $a_i=\prod_{j=1}^it_j$. We have then 
	$$J(\textbf{t},f_0)=
	\Delta(\textbf{t})I(\textbf{t},\phi_0),$$
	where (agreeing that $a_0=1$)
	$$\Delta(\textbf{t})={\zeta(-1)}^{\sum_{j \not \equiv r({\rm mod} 2)}v(a_j)}|\prod_{i=1}^{r-1}a_i|^{-1/2}\prod_{j\not \equiv r ({\rm mod} 2)}\gamma(a_ja_{j-1}^{-1},\psi).$$ 
\end{prop}
Here $\zeta :k^*\to\{\pm1\}$ is the non-trivial quadratic character and $\gamma(\bullet,\psi)$ is the Weil constant defined by: given a compact open neighbourhood $\Omega$ of 0 in $F$, for $|a|$ large enough, we have:
$$\int_{\Omega}\psi(\frac{ax^2}{2})dx=|a|^{-1/2}\gamma(a,\psi).$$
\begin{prop}[cf. \cite{VC2}]\label{result3}
	The proposition \ref{result} is still true when $F$ is a local field of characteristic zero, with sufficiently large residue characteristic.
\end{prop}
Similar identities are expected to be true for other relevant orbits.

From now on, we focus only on the case where $F$ is a local non-archimedean field of characteristic zero and the residual characteristic of $F$ is larger than $2r+1$ (the condition for the residual characteristic is needed in our calculation, cf. propositions \ref{prop1} and \ref{charac2}). For $\mathbf{t}=\D(a_ja_{j-1}^{-1},1\leq j\leq r)\in T_r(F)$, we introduce two new factors:
$$\Delta_r(\textbf{t})={\zeta(-1)}^{\sum_{j \not\equiv r({\rm mod} 2)}v(a_j)}|\prod_{i=1}^{r-1}a_i|^{-1/2}\prod_{j\not\equiv r ({\rm mod} 2)}\left[\gamma(a_ja_{j-1}^{-1},\psi)(a_ja_{j-1}^{-1},\varpi)^{v(a_ja_{j-1}^{-1})}\right]$$  
and
$$\Delta'_r(\textbf{t})={\zeta(-1)}^{\sum_{j \equiv r({\rm mod} 2)}v(a_j)}|\prod_{i=1}^{r-1}a_i|^{-1/2}\prod_{j\equiv r ({\rm mod} 2)}\left[\gamma(a_ja_{j-1}^{-1},\psi)(a_ja_{j-1}^{-1},\varpi)^{v(a_ja_{j-1}^{-1})}\right].$$ 
Here, $(,):F^*\times F^*\to \{\pm 1\}$ is the Hilbert symbol. It is a bilinear form on $F^*$ that defines a nondegenerate bilinear form on $F^*/(F^*)^2$ and satisfies $$(x,-x)=(x,y)(y,x)=1.$$

 Our main result is:
 
\begin{theo}[The main theorem]\label{thm1}
	Let $\phi$ be a smooth function of compact support over $S_r(F)$ and $f$ be a smooth genuine function of compact support over $\tG_r(F)$. If $\phi$ and $f$ satisfy 
	$J(\textbf{t},f)=
		\Delta_r(\textbf{t})I(\textbf{t},\phi)$
		(resp. $J(\textbf{t},f)=
		\Delta'_r(\textbf{t})I(\textbf{t},\phi)
	$)
	for all $\mathbf{t}\in T_r(F)$, then for all $w\in W_r^R$ and all $\mathbf{t}\in T_w(F)$ there exists $\Delta_{w}(\mathbf{t})$ (resp. $\Delta'_{w}(\mathbf{t})$) (we refer to Theorem \ref{main2} for the formula of  $\Delta_{w}(\mathbf{t})$ and of $\Delta'_{w}(\mathbf{t})$) such that  
	$J(w{\bf t},f)=
	\Delta_w(\textbf{t})I(w\textbf{t},\phi)$ (resp. $J(w{\bf t},f)=\Delta'_w(\textbf{t})I(w\textbf{t},\phi)$).	 
\end{theo}
It should mention that to calculate $\Delta_w$ (and $\Delta'_w$), we must use both $\Delta_i$ and $\Delta'_i$. 

Let $g\in \G_r(F)$. We denote by $a_i(g)$ the determinant of the sub-matrix made of the first $i$ lines and of the first $i$ columns of the matrix $g$. We denote by $\phi_1$ the function  $$\phi_1:\G_r(F)\to\{\pm 1\};\quad g\mapsto \prod_{j\not \equiv r\mathrm{mod}\,2}(a_j(g)a_{j-1}(g)^{-1},\varpi)^{v(a_j(g)a_{j-1}(g)^{-1})}.$$ 
Note that $a_{i}({}^tngn)=a_i(g)$ for all $1\leq i\leq r$. Using Propositions \ref{result} and \ref{result3} we have 
$$J(\mathbf{t},f)=\Delta_r(\mathbf{t})I(\mathbf{t},\phi_0.\phi_1)$$for all $\mathbf{t}=\D(a_ja_{j-1}^{-1},1\leq j\leq r)$.

Applying the main theorem for $\phi=\phi_0.\phi_1$ and $f=f_0$, then the (suggested) transfer factors of any kind of relevant orbits are calculated. 
\begin{coro} Let $F$ be a local field of characteristic zero with sufficiently large residue characteristic. Let $w\mathbf{t}$ be a relevant orbit of $\G_r$. We have then
	$$J(w\mathbf{t},f_0)=\Delta_w(\mathbf{t})\phi_1(w\mathbf{t})I(w\mathbf{t},\phi_0).$$
\end{coro} 


The integral $J$ above is in fact the Kloosterman integral which is considered in \cite{J3,J2} (see. section \ref{sec1}). So we have the following density theorem for Kloosterman integral: 
\begin{prop}[cf. \cite{J2}]\label{thm2} If the diagonal orbital integral $J(\mathbf{t},f)$ of a function $f\in \mathcal{C}_c^\infty(\tG_r(F))$ vanishes for all $\mathbf{t}\in T_r(F)$, then all the orbital integrals $J(w\mathbf{t},f)$ with $w\in W^R_r$ and $\mathbf{t}\in T_w(F)$ of $f$ vanish.	
\end{prop}

Assume that a function $\phi\in\mathcal{C}_c^\infty(S_r(F))$ satisfies $I(\mathbf{t},\phi)=0$ for all $\mathbf{t}\in T_r(F)$. We have then
$$I(\mathbf{t},\phi)=\Delta_r(\mathbf{t})J(\mathbf{t},0)$$ for all $\mathbf{t}\in T_r(F)$. Applying the main theorem for the couple of functions $(\phi,0)$, we obtain the following density theorem:
\begin{prop}
	If the diagonal orbital integrals $I(\mathbf{t},\phi)$ of a function $\phi\in \mathcal{C}_c^\infty(S_r(F))$ vanishes for all $\mathbf{t}\in T_r(F)$, then all the orbital integrals $I(w\mathbf{t},\phi)$ with $w\in W^R_r$ and $\mathbf{t}\in T_w(F)$ of $\phi$ vanish.
\end{prop}

We now state the organization of this manuscript. The main tool we use to prove the main theorem (Theorem \ref{thm1}) is Shalika germs which describe the asymptotic behavior of the orbital integrals. In Section 2 (resp. 3) we recall this asymptotic behavior of the integral $J$ (resp. the integral $I$) and do some calculation for the germ functions. The main theorem is proved by induction in Section 4. In this section, we introduce some intermediate integrals which are designed to use inductive argument. The germ relations are used to handle the cases when we can not use directly these intermediate integrals. 

The proof of the main theorem closely follows the guidelines of \cite{J2}, the new ingredient is the occurrence of the new factors $\Delta_r$ and $\Delta'_r$. 

In the following sections, we need (to recall) some properties of the Hilbert symbol and of the Weil constant.
\begin{prop}[cf. {\cite[Proposition 1]{M}, \cite{VC1}}] For $a,b,c\in F^*$ we have:
	\begin{enumerate}
		\item $(a,b)=(b,a)$.
		\item $(a,bc)=(a,b)(a,c)$.
		\item If $v(a),v(b)$ are even, then $(a,b)=1$.
		\item If $v(a)$ is even, $v(b)$ is odd, then $(a,b)=(a,\varpi)$.
		\item If $v(a),v(b)$ are odd, then $(a,b)=(-ab,\varpi)$.
		\item If $a\in k^*$, then $(a,\varpi)=\zeta(a)$.
	\end{enumerate}
\end{prop}
\begin{prop}[cf. {\cite[Proposition 2]{M}}] Let $\psi$ be a non-trivial additive character of level 0. We have:
	\begin{enumerate}
		\item $\gamma(a,\psi)=1$ if $v(a)$ is even.
		\item $\gamma(ab,\psi)=\gamma(a,\psi)\gamma(b,\psi)(a,b)$.
	\end{enumerate}
\end{prop}


\section{Computation of the germ on $J$ side}\label{sec1}
For a convenience, we shall rewrite the orbital integral $J$. 
Suppose that $w\in W_r^R$ and $\mathbf{t}\in T_w$. Then let $P_w=M_wN_w$ be the standard parabolic subgroup which has Levi factor $M_w$. Let $V_w=N_r\cap M_w$. We have $$(N_r\times N_r)_{w_{\G_r}w\mathbf{t}}=N_r^{w}:=\{(n_1,n_2)|w_{\G_r}n_1^{-1}w_{\G_r}wn_2=w\}$$ for all $\mathbf{t}\in T_w$. Furthermore, if $(n_1,n_2)\in N_r^w$ then $n_2,{}^t(n^{w_{\G_r}}_1):={}^t(w_{\G_r}n_1w_{\G_r})\in V_{w}$ and $n_2=wn^{w_{\G_r}}_1w$. It implies that any point of the orbit of $w_{\G_r}w\mathbf{t}$ under the action of $N_r\times N_r$ can be uniquely written in the following form
$$\mu(u_1,u_2,v)=w_{\G_r}{}^tu_1w\mathbf{t}vu_2$$
with $u_i\in N_w(F)$ and $v\in V_w(F)$. Note that the map $\mu$ is an isomorphism of analytic varieties (over $F$) $N_w(F)\times N_w(F)\times V_w(F)$ onto the orbit of $w\mathbf{t}$.  Thus 
$$J(w\mathbf{t},f)=\int_{N_w(F)\times N_w(F)\times V_w(F)}f((w_{\G_r}{}^tu_1w\mathbf{t}vu_2,1))\theta(u_1u_2v)du_1dvdu_2.$$
Denote by $f'$ the function $g\mapsto f((w_{\G_r}g,1)),\,\forall g\in \G_r$, then $f'\in \mathcal{C}_c^\infty(\G_r(F
))$. The integral $J$ is then the orbital Kloosterman integral $\mathrm{Kloos}(w\mathbf{t};f')$ which is considered in \cite{J3} (in loc. cit., it is denoted by $I(w\mathbf{t};f')$). 
This integral converges and defines a smooth function on $T_w(F)$.

We let $M$ be the standard Levi-subgroup of type $(r-1,1)$. The corresponding element $w_M$ is $\left(\begin{smallmatrix}
w_{\G_{r-1}}&0\\
	0& 1
\end{smallmatrix}\right)$. We denote by $T_{w_M}^{w_{\G_r}}$ the set of matrices $\mathbf{t}\in T_{w_M}(F)$ such that $\det(\mathbf{t})=\det(w_M)\det(w_{\G_r})$. There exists a smooth function $K_{w_M}^{w_{\G_r}}$ on $T_{w_M}^{w_{\G_r}}$ (cf. \cite{J3,J2}) with the following property: for any $f\in \mathcal{C}_c^{\infty}(\G_r(F))$ there is a smooth function of compact support $\omega_f$ on $T_M$ such that 
\begin{equation*}
	J(w_M\mathbf{t},f)=\omega_f(\mathbf{t})+\sum_{\alpha\beta=\mathbf{t}}K_{w_M}^{w_{\G_r}}(\alpha)J(w_{\G_r}\beta,f).
\end{equation*}
Here, we still denote by $f$ the genuine function in $\mathcal{C}_c^{\infty}(\tG_r(F))$ defined by $(g,z)\mapsto zf(g)$ with $f\in \mathcal{C}_c^{\infty}(\G_r(F))$. The sum is over all pairs in $$\{(\alpha,\beta)\in (T_{w_M}^{w_{\G_r}},T_{\G_r}(F))\,| \,\alpha\beta=\mathbf{t} \}.$$
The function $K_{w_M}^{w_{\G_r}}$ is the \textbf{germ} (for the side $J$) along the subset $T_{w_M}^{w_{\G_r}}$. It is not unique.

 
\begin{prop}[cf. {\cite[Proposition 3.1]{J2}}]\label{prop1}Suppose that the residual characteristic of $F$ is larger than $r$. 
Let $$\alpha=\D(a,\dots,a,a^{1-r}\det(w_Mw_{\G_r})).$$
Then, for $|a|$ sufficiently small,

\begin{equation*}
	K_{w_M}^{w_{\G_r}}(\alpha)=|a|^{-\frac{(r-1)^2}{2}}\psi\left(\frac{r}{2a}\right)\left(\frac{r}{2^{r-1}},a^{-1}\right)\gamma(a^{-1},\psi)^{r-1}.
\end{equation*}
\end{prop}
Note that the characteristic of $k$ is larger than $r$, we have then $\frac{i+1}{2i}\in \mathcal{O}^*$ for all $1\leq i\leq r-1$. Since in our case $\theta(n)=\psi(\frac{1}{2}\sum_{i=2}^rn_{i-1,i})$ (is not $\theta(n)=\psi(\sum_{i=2}^rn_{i-1,i})$ in loc. cit.), our obtained formula (using the same processing) is a bit different from the one of Jacquet (in loc. cit.). 

\section{Computation of the germ on $I$ side}
The discussion of \cite[\S 2]{J3} applies to our situation where $\G_r(F)$ is replaced by $S_r(F)$ and the group $N_r\times N_r$ ($N_r\times N_r$ acts on $\G_r$ by $g\mapsto {}^tngn'$) by the group $N_r(F)$ acting on $S_r(F)$. Suppose that $w\in W_r^R$ and $\mathbf{t}\in T_w$. Then let $P_w=M_wN_w$ be the standard parabolic subgroup which has Levi factor $M_w$. Let $V_w=N_r\cap M_w$. We have $$(N_r)_{w\mathbf{t}}=(N_r)_{w}:=\{n\in N_r|{}^{t}nwn=w\}$$ for all $\mathbf{t}\in T_w$. Furthermore, $(N_r)_w$ is the set of $n\in V_{w}$ such that ${}^{t}nwn=w$. 

Any element of the orbit of $w\mathbf{t}$  can be  written in the form
$${}^tu{}^t vw\mathbf{t}vu$$
with $u\in N_w(F)$ and $v\in V_w(F)$. Denote by $v_1=w{}^tvwv$. We have then $v_1$ is an element of $V^1_w:=\{v\in V_w|w{}^tvw=v\}$. Thus in fact any point of the orbit of $w\mathbf{t}$ can be written uniquely in the form
$$\nu(u,v_1):={}^tuw\mathbf{t}v_1u$$
with $u\in N_w(F)$ and $v_1\in V^1_w(F)$. The map $\nu$ is a diffeomorphism of $N_r(F)\times V^1_w(F)$ onto the orbit $w\mathbf{t}$ in $S_r(F)$. Thus
$$I(w\mathbf{t},\phi)=\int_{N_w(F)\times V^1_w(F)}\phi({}^tuw\mathbf{t}v_1u)\theta(u^2v_1)dudv_1.$$

Now, we let $M$ be the standard Levi-subgroup of type $(r-1,1)$. As before, there exists a smooth function $L_{w_M}^{w_{\G_r}}$ on $T_{w_M}^{w_{\G_r}}$  with the following property: for any $\phi\in \mathcal{C}_c^\infty(S_r(F))$ there is a smooth function of compact support $\omega_\phi$ such that
$$I(w_M\mathbf{t},\phi)=\omega_\phi(\mathbf{t})+\sum_{\alpha\beta=\mathbf{t}}L_{w_M}^{w_{\G_r}}(\alpha)I(w_{\G_r}\beta,\phi).$$
The sum is over all pairs in $\{(\alpha,\beta)\in (T_{w_M}^{w_{\G_r}},T_{\G_r}(F))\,| \,\alpha\beta=\mathbf{t} \}.$
The function $L_{w_M}^{w_{\G_r}}$ is the \textbf{germ} (for the side $I$) along the subset $T_{w_M}^{w_{\G_r}}$. It is not unique. 

Let $\mathbf{t}=\D(a,\dots,a,a^{1-r}\det(w_Mw_{\G_r}))$. Since $\omega_\phi$ is a smooth function of compact support, we can choose $|a|$ small enough such that $\omega_\phi(\mathbf{t})=0$. We consider the pair $(\alpha,\beta)\in (T_{w_M}^{w_{\G_r}},T_{\G_r}(F))$ such that $\alpha\beta=\mathbf{t}$. Since $\det(\alpha)=\det(w_Mw_{\G_r})=\det(\mathbf{t})$ (by definition), we have $\det(\beta)=1$. Moreover, $\beta=\D(z,z,\dots,z)$ with $z^{r}=1$.
 
We denote by $[x]$ the integral part of a real number $x$. Let $K_m:=\mathrm{Id}_r+\varpi^m\mathfrak{gl}_r(\0)$ be the principal congruence subgroup of $\G_r(F)$. We denote by $c_1(r)$ the scalar 
$$c_1(r):=\mathrm{vol}(\varpi^m\0)^{-\left[\frac{r^2}{4}\right]}.$$
Let $\phi$ be the product of the characteristic function of $w_{\G_r}K_m\cap S_r(F)$ and the scalar $c_1(r)$. We have the following lemma: 
\begin{lemm} Let $\beta=\D(z,z,\dots,z)$ with $z^r=1$ and $\phi$ as above. For $m$ large enough, we have then  
	$$I(w_{\G_r}\beta,\phi)=\begin{cases}
		1,&\text{if } z=1,\\
		0,&\text{otherwise}. 
	\end{cases}$$
\end{lemm} 
\begin{proof}
	Firstly, we calculate the integral $I(w_{\G_r},\phi)$. After a unimodular change of variables, this integral has the form
	$$I(w_{\G_r},\phi)=\int\phi(x)\psi\left(\frac{\sum_{i+j=r+2}x_{i,j}}{2}\right)\otimes dx_{i,j}$$
	where  $x=(x_{i,j})$ is a symmetric matrix such that 
	$$x_{i,j}=\begin{cases}
	0,&\text{if } i+j<r+1,\\
	1,&\text{if } i+j=r+1.
	\end{cases}$$
    The variables are the entries $x_{i,j}\in F$ with $i+j\geq r+2,\,i<j$, the entries $x_{i,i}\in F$ with $2i\geq r+2$. 
    
    The number of entries $x_{i,j}$ with $i+j\geq r+2$ is $\frac{r(r-1)}{2}$. The number of entries $x_{i,i}$ with $2i\geq r+2$ is $r-\left[\frac{r+3}{2}\right]+1=\left[\frac{r}{2}\right]$. So the number of variables of above integral is $\frac{\frac{r(r-1)}{2}-\left[\frac{r}{2}\right]}{2}+\left[\frac{r}{2}\right]=\left[\frac{r^2}{4}\right]$.
    
    Now we take $\phi$ be a product of the characteristic function of $w_{\G_r}K_m$ and the scalar $c_1(r)$, this integral is equal to 
    $$c_1(r)\int\psi\left(\frac{\sum_{i+j=r+2}x_{i,j}}{2}\right)\otimes dx_{i,j}$$
    integrated over the domain:
    \begin{align*}
    	x_{i,j}&\equiv 0 \mod \varpi^m\0\,\text{for } i+j\geq r+2.
    \end{align*}
    
    Moreover, since $\psi$ is of order $0$, we have $\psi\left(\frac{\sum{x_{i,j}}}{2}\right)=1$. It implies that 
    $$\int\psi\left(\frac{\sum_{i+j=r+2}x_{i,j}}{2}\right)\otimes dx_{i,j}=\mathrm{vol}(\varpi^m\0)^{\left[\frac{r^2}{4}\right]}=c_1(r)^{-1}.$$
    In consequence, the first assertion is proved.
    
    Choosing $m$ large enough such that $z \not \in 1+\varpi^m\0$ for all $z$ which satisfy $z^r=1$ and $z\neq 1$. We have then ${}^tnw_{\G_r}\beta n\not\in w_{\G_r}K_m,\forall n\in N_r(F)$. It follows the second assertion.
\end{proof}
As a consequence, we obtain that  
\begin{equation}\label{i1}
L_{w_M}^{w_{\G_r}}(\alpha)=I(w_M\alpha,\phi),
\end{equation}
where $$\alpha=\D(a,\dots,a,a^{1-r}\det(w_Mw_{\G_r})),$$
and $|a|$ is small enough.

Our goal of this section is to compute the germ $L_{w_M}^{w_{\G_r}}$, or, what amounts to the same, the integral $I(w_M\alpha,\phi)$ where $\phi$  is the function defined above.

Let $P$ be the parabolic subgroup of type $(r-1,1)$ and $N_M$ its unipotent radical. Then $P=MN_M$. We denote by $V_M=N_r\cap M$ and $V_M^1=\{v\in V_M|w_M{}^tvw_M=v\}$. If $u\in V_M$ then $u=(u_{i,j})$ of the following form
\begin{eqnarray*}
u_{i,j}&=&0 \text{ for $(i,j)\not \in\{(i,i),(i,r)|1\leq i\leq r\}$},\\
u_{i,i}&=&1 \text{ for $1\leq i\leq r$}.
\end{eqnarray*}
If $v\in V_M^1$, then $v=(v_{i,j})$ of the following form
\begin{eqnarray*}
	v_{i,r}&=&0 \text{ for $1\leq i\leq r-1$},\\
	v_{i,i}&=&1 \text{ for $1\leq i\leq r$},\\
	v_{i,j}&=&0 \text{ for $i>j$},\\
	v_{i,j}&=&v_{(r-j),(r-i)} \text{ for $1\leq i<j\leq r-1$}.
\end{eqnarray*}
The orbital integral $I(w_M\alpha,\phi)$ can be written as follows
\begin{equation}\label{eq1}
I(w_M\alpha,\phi)=\int \phi({}^tuw\mathbf{t}vu)\theta(u^2v)(\otimes_{i=1}^{r-1} du_{i,r})\otimes dv_{i,j},
\end{equation}
The variables are the entries $u_{i,r}\in F$ with $1\leq i\leq r-1$, the entries $v_{i,r-i}\in F$ with $1\leq i<r-i$, the entries $v_{i,j}\in F$ with $1\leq i<j\leq r-1,\,i\leq r-j$.
Denote by $c_2(r)$ the number of variables of above integral. We have then 
$$c_2(r)=(r-1)+\left[\frac{r-1}{2}\right]+\frac{1}{2}\left(\frac{(r-2)(r-1)}{2}-\left[\frac{r-1}{2}\right]\right)=\left[\frac{r^2+2r-3}{4}\right].$$
After a unimodular change of variables, (\ref{eq1}) can be written as 
\begin{equation*}
	I(w_M\alpha,\phi)=|a|^{-c_2(r)}\int \phi(x)\psi\left(\frac{\sum_{i=1}^rx_{i,r+1-i}}{2a}\right)\otimes dx_{i,j},
\end{equation*}
where $x=(x_{i,j})$ denotes a matrix of the following form 
\begin{eqnarray*}
x_{i,j}&=&0 \text{ for } i+j<r,\\
x_{i,j}&=&a \text{ for } i+j=r,\\
x_{i,j}&=&x_{j,i}.
\end{eqnarray*}
The variables are the entries $x_{i,j}\in F$ with $i+j\geq r+1,\,i<j$, the entries $x_{i,i}\in F$ with $2i\geq r+1$, except the entry $x_{r,r}$ which is a dependent variable. The entry $x_{r,r}$ can be computed (from the condition that the determinant of the matrix $x$ be $\det(w_{\G_r})$) by 
$$a^{r-1}\det(w_M)x_{r,r}+\det\left(\begin{matrix}
0&\cdots&0&a&x_{1,r}\\
\vdots&\iddots&\iddots&\iddots&\vdots\\
0&\iddots&\iddots&x_{r-2,r-1}&x_{r-2,r}\\
a&\iddots&x_{r-1,r-2}&x_{r-1,r-1}&x_{r-1,r}\\
x_{r,1}&\dots&x_{r,r-2}&x_{r,r-1}&0\end{matrix}\right)=\det(w_{\G_r}).$$

Let $\mathcal{I}$ be a subset of $\{(i,j)|1\leq i\leq j\leq r\}$. We denote by $x^{\mathcal{I}}$ the matrix obtained from $x$ by replacing the entries $x_{i,j}$ and $x_{j,i}$ by $0$ with $(i,j)\in \mathcal{I}$. For instance, the above condition of $x_{r,r}$ can be written as follows
$$a^{r-1}\det(w_M)x_{r,r}+\det(x^{\{(r,r)\}})=\det(w_{\G_r}).$$

Since $\phi$ is the product of the characteristic function of $w_{\G_r}K_m$ and the scalar $c_1(r)$, the above integral is equal to 
$$|a|^{-c_2(r)}c_1(r)\int \psi\left(\frac{\sum_{i=1}^rx_{i,r+1-i}}{2a}\right)\otimes dx_{i,j}$$ 
integrated over the set: 
\begin{align*}
&x_{i,j}\equiv 1 \mod \varpi^m\0\, \text{for } i+j=r+1,\\
&x_{i,j}\equiv 0 \mod \varpi^m\0\, \text{for } i+j>r+1,\,j>i, \,(i,j)\neq (r,r),\\
&x_{r,r}\equiv 0 \mod \varpi^m\0.
\end{align*}

The last condition can be written as follows
$$\det(x^{\{(r,r)\}})\equiv \det(w_{\G_r})\mod a^{r-1}\varpi^m\0.$$
By single out the variable $x_{2,r}$ ($=x_{r,2}$) from the left hand side, we have:
\begin{equation}\label{squareroot}
2x_{2,r}ay+x_{2,r}^2a^2z+\det(x^{\{(r,r),(2,r)\}})\equiv \det(w_{\G_r})\mod a^{r-1}\varpi^m\0,
\end{equation}
where $y,z\in 1+\varpi^m\0$. Both $y$ and $z$ depend only on the variables $x_{i,j}$ with $(i,j)\neq (2,r)$. 

We denote $T:=\det(x^{\{(r,r),(2,r)\}})-\det(w_{\G_r})$. Since $x_{2,r}\equiv 0 \mod \varpi^m\0$, we have $T\equiv 0 \mod a\varpi^m\0$. We have then $y^2-zT\in 1+\varpi^m\0$. It implies that $y^2-zT$ has a square root in $F$ for $m$ large enough. Assume that $\mu$ is a square root of $y^2-zT$. Since $y^2-zT\in 1+\varpi^m\0$, we have $(\mu-1)(\mu+1)\in \varpi^m\0$. It implies that $\max\{v(\mu-1),v(\mu+1)\}\geq \frac{m}{2}$. We can suppose that $v(\mu-1)=\max\{v(\mu-1),v(\mu+1)\}$ (otherwise, we change $\mu$ to $-\mu$). We get then $\mu=1+ \varpi^{m/2}\0$ and $\mu+1=2+1+ \varpi^{m/2}\0\in\0^*$. Thus $v(\mu-1)\geq m$, i.e $\mu\in 1+\varpi^m\0$. We have proved that there exists a square root $\mu$ of $y^2-zT$ such that 
$\mu\in 1+\varpi^m\0$.

The condition (\ref{squareroot}) can be read
$$a^2z\left(x_{2,r}-\frac{-y-\mu}{az}\right)\left(x_{2,r}-\frac{-y+\mu}{az}\right)\equiv 0 \mod a^{r-1}\varpi^m\0.$$

For $|a|$ small enough (i.e the valuation of $a$ large enough), from the definition of $y,z$ and $T$, we have $\frac{-y+\mu}{az}\in a^{-1}\varpi^m\0$ and $\frac{-y-\mu}{az} \in -2a^{-1}+a^{-1}\varpi^m\0 $. Thus above condition is equivalent to 
$$x_{2,r}\equiv \frac{-y+\mu}{az}\mod a^{r-2}\varpi^m\0.$$



Using this condition, we can integrate the variable $x_{2,r}$ away from the orbital integral $I$ to obtain the scalar factor $|a|^{r-2}\mathrm{vol}(\varpi^m\0)$ multiple with a new integral. The new integral has the same form of the old one but the domain of integration is defined by:
\begin{align*}
x_{i,j}&\equiv 1 \mod \varpi^m\0\, \text{for } i+j=r+1,\\
x_{i,j}&\equiv 0 \mod \varpi^m\0\, \text{for } i+j>r+1, \,(i,j)\neq (r,r),(2,r)\\
T&\equiv 0 \mod a\varpi^m\0.
\end{align*}

The determinant of $x^{\{(r,r),(2,r)\}}$ has the form 
$$ \prod_{i+j=r+1}{x_{i,j}}\det(w_{\G_r})+aX$$
with $X\in \varpi^m\0$. Thus the condition on $T$ can read
$$\prod_{i+j=r+1}x_{i,j}\equiv 1 \mod a\varpi^m\0.$$

Now, we integrate over the variables $x_{i,j}$ with $i+j>r+1,\,(i,j)\neq (r,r),(2,r)$ and we get 
\begin{equation}\label{i2}
I(w_M\alpha,\phi)=|a|^{r-2-c_2(r)}c_1(r)\mathrm{vol}(\varpi^m\0)^{c_2(r)-\left[\frac{r+1}{2}\right]+1}I(a,r),
\end{equation}
where the function $I(a,r)$ is defined as follows.

If $r=2\ell$, then
\begin{equation*}
I(a,r):=\mathrm{vol}(\varpi^m\0)^{-1}\int\psi\left(\frac{x_1+x_2+\dots+x_\ell}{a}\right)\otimes dx_i.
\end{equation*}
The domain of integration is defined by 
\begin{align*}
	x_i&\equiv 1\mod \varpi^m\0\\
	x_1^2x_2^2\dots x_\ell^2&\equiv 1 \mod a\varpi^m\0.
\end{align*}

If $r=2\ell+1$, then
\begin{equation*}
I(a,r):=\mathrm{vol}(\varpi^m\0)^{-1}\int\psi\left(\frac{2x_1+2x_2+\dots+2x_\ell+x_{\ell+1}}{2a}\right)\otimes dx_i.
\end{equation*}
The domain of integration is defined by 
\begin{align*}
x_i&\equiv 1\mod \varpi^m\0\\
x_1^2x_2^2\dots x_\ell^2x_{\ell+1}&\equiv 1 \mod a\varpi^m\0.
\end{align*}

We have:
\begin{prop}\label{charac2}
	Suppose that the residual characteristic of $F$ is larger than $2r+1$. Then, if $m$ is large enough:
	$$I(a,r)=|a|^{\frac{1}{2}{\left[\frac{r-1}{2}\right]}+1}\psi\left(\frac{r}{2a}\right)\left(\frac{r}{2^{r-1}},a^{-1}\right)\gamma\left(a^{-1},\psi\right)^{\left[\frac{r-1}{2}\right]},$$
	for $|a|$ sufficiently small. 
\end{prop}
\begin{proof}
	Firstly, we consider the case $r=2\ell+1$. We change variables 
	$$x_{\ell+1}=\left(\prod_{i=1}^{\ell}x_i^{-2}\right)t$$
	with $t\in 1+a\varpi^m\0$ and integrate over $t$. We obtain
	$$I(a,2\ell+1)=|a|\int \psi\left(\frac{\delta}{2a}\right)\otimes dx_i,$$
	where the phase function $\delta$ is given by: 
	$$\delta=\sum_{i=1}^{\ell}2x_i+\frac{1}{\prod_{i=1}^{\ell}x_i^2}.$$
	We set $x_i=1+u_i$ with $u_i\in \varpi^m\0$. This function can be written as
	$$\delta=2\ell+\sum_{i=1}^{\ell}2u_i+\prod_{i=1}^{\ell}\frac{1}{1+2u_i+u_i^2}.$$
	Now we consider the Taylor expansion of this function at the origin. It has the form:
	$$(2\ell+1) +3\sum_{i=1}^{\ell}u_i^2+4\sum_{1\leq i<j\leq \ell}u_iu_j+\text{higher digree terms}.$$
	
	Since the quadratic form $$3\sum_{i=1}^{\ell}X_i^2+4\sum_{1\leq i<j\leq \ell}X_iX_j$$ is equivalent to the quadratic form
	$$\sum_{i=1}^{\ell}\frac{2i+1}{2i-1}Y_i^2$$ by a unipotent transformation (see. \cite[Lemma 5.1]{J2}), after a unimodular change variables, this Taylor expansion may be written in the form:
	$$\delta=(2\ell+1)+\sum_{i=1}^{\ell}\frac{2i+1}{2i-1} y_i^2+\text{higher degree terms}.$$
	
	We choose $m$ large enough such that the origin is the only critical point in the domain of integration. Using the principle of stationary phase, there exists a neighbourhood $0\in\Omega$ in $F$ such that for $|a|$ small enough $I(a,2\ell+1)$ is the product of the factors:
	$$|a|\psi\left(\frac{2\ell+1}{2a}\right),$$
	$$\int_{\Omega}\psi\left(\frac{2i+1}{2(2i-1)a}y_i^2\right)dy_i=\left|\frac{2i+1}{(2i-1)a}\right|^{-1/2}\gamma\left(\frac{2i+1}{(2i-1)a},\psi\right),\,1\leq i\leq \ell.$$
	
	Using the property of Weil constant that $\gamma(bc,\psi)=\gamma(b,\psi)\gamma(c,\psi)(b,c)$, we have: 
	$$\gamma\left(\frac{2i+1}{(2i-1)a},\psi\right)=\gamma\left(\frac{2i+1}{2i-1},\psi\right)\gamma(a^{-1},\psi)\left(\frac{2i+1}{2i-1},a^{-1}\right).$$
		
	Since the residual characteristic of $F$ larger than $2r+1$, we have $\frac{2i+1}{2i-1}\in \0^*$ for all $1\leq i\leq \ell$. So above equation simplifies to
	$$\gamma\left(\frac{2i+1}{(2i-1)a},\psi\right)=\gamma(a^{-1},\psi)\left(\frac{2i+1}{2i-1},a^{-1}\right).$$ 
	
	Using the formula $(b,c)(b',c)=(bb',c)$ and $\frac{2i+1}{2i-1}\in\0^*$ for all $1\leq i\leq \ell$ , we have: 
	\begin{eqnarray*}
	I(a,2\ell+1)&=&|a|^{\frac{\ell}{2}+1}\psi\left(\frac{2\ell+1}{2a}\right)(2\ell+1,a^{-1})\gamma(a^{-1},\psi)^{\ell}\\
	&=&|a|^{\frac{\ell}{2}+1}\psi\left(\frac{2\ell+1}{2a}\right)(2\ell+1,a^{-1})\gamma(a^{-1},\psi)^{\ell}\left(\frac{1}{2},a^{-1}\right)^{2\ell}\\
	&=&|a|^{\frac{\ell}{2}+1}\psi\left(\frac{2\ell+1}{2a}\right)\left(\frac{2\ell+1}{2^{2\ell}},a^{-1}\right)\gamma(a^{-1},\psi)^{\ell}.
	\end{eqnarray*}

	For the case $r=2\ell$. We set
	$$x_\ell=t\left(\prod_{i=1}^{\ell-1}x_i\right)^{-1}.$$
	Since $\prod_{i=1}^{\ell}x_i^2\equiv 1 \mod a\varpi^m\0$, we have $t^2\equiv 1\mod a\varpi^m\0$. For $m$ large enough and $|a|$ small enough, this condition is equivalent to $t\equiv \pm 1 \mod a\varpi^m\0$. Moreover $t=\prod_{i=1}^\ell x_i\equiv 1 \mod \varpi^m\0$, so $t\equiv 1 \mod a\varpi^m\0$. We now integrate over $t$ to get 
	$$I(a,2\ell)=|a|\int\psi\left(\frac{\delta}{a}\right)\otimes dx_i,$$
	where the phase function $\delta$ is given by: 
	$$\delta=\sum_{i=1}^{\ell-1}x_i+\prod_{i=1}^{\ell-1}\frac{1}{x_i}.$$
	
	We set $x_i=1+u_i$ with $u_i\in \varpi^m\0$. This function can be written as:
	$$\delta=\ell-1+\sum_{i=1}^{\ell-1}u_i+\prod_{i=1}^{\ell-1}\frac{1}{1+u_i}.$$
	The Taylor expansion of this function at the origin has the form:
	$$\ell+\sum_{i=1}^{\ell-1}u_i^2+\sum_{1\leq i<j\leq \ell-1}u_iu_j+\text{ higher degree terms}.$$
	
	Since the quadratic form $$\sum_{i=1}^{\ell-1}X_i^2+\sum_{1\leq i<j\leq \ell-1}X_iX_j$$ is equivalent to the quadratic form
	$$\sum_{i=1}^{\ell-1}\frac{i+1}{2i}Y_i^2$$ by a unipotent transformation (see. \cite[Lemma 2.1]{J2}), after a unimodular change variables, this Taylor expansion may be written in the form:
	$$\delta=\ell+\sum_{i=1}^{\ell-1}\frac{i+1}{2i} y_i^2+\text{higher degree terms}.$$
	
	We choose $m$ large enough such that the origin is the only critical point in the domain of integration. Using the principle of stationary phase, there exists a neighbourhood $0\in\Omega$ in $F$ such that for $|a|$ small enough $I(a,2\ell)$ is the product of the factors:
	$$|a|\psi\left(\frac{\ell}{a}\right),$$
	$$\int_{\Omega}\psi\left(\frac{i+1}{2ai}y_i^2\right)dy_i=\left|\frac{i+1}{ai}\right|^{-1/2}\gamma\left(\frac{i+1}{ai},\psi\right),\,1\leq i\leq \ell-1 .$$
	Using the property of Weil constant that $\gamma(bc,\psi)=\gamma(b,\psi)\gamma(c,\psi)(b,c)$, we have: 
	$$\gamma\left(\frac{i+1}{ai},\psi\right)=\gamma\left(\frac{i+1}{i},\psi\right)\gamma(a^{-1},\psi)\left(\frac{i+1}{i},a^{-1}\right).$$
	
	Since the residual characteristic of $F$ larger than $2r+1$, we have $\frac{i+1}{i}\in \0^*$ for all $1\leq i\leq \ell-1$. So above equation simplifies to
	$$\gamma\left(\frac{i+1}{ai},\psi\right)=\gamma(a^{-1},\psi)\left(\frac{i+1}{i},a^{-1}\right).$$ 
	
	Using the formula $(b,c)(b',c)=(bb',c)$ and $\frac{i+1}{i}\in\0^*$ for all $1\leq i\leq \ell-1$, we have: 
	\begin{eqnarray*}
	I(a,2\ell)&=&|a|^{\frac{\ell+1}{2}}\psi\left(\frac{\ell}{a}\right)(\ell,a^{-1})\gamma(a^{-1},\psi)^{\ell-1}\\
	&=&|a|^{\frac{\ell+1}{2}}\psi\left(\frac{\ell}{a}\right)(\ell,a^{-1})\gamma(a^{-1},\psi)^{\ell-1}\left(\frac{1}{2},a^{-1}\right)^{2\ell-2}\\
	&=&|a|^{\frac{\ell-1}{2}+1}\psi\left(\frac{2\ell}{2a}\right)\left(\frac{2\ell}{2^{2\ell-1}},a^{-1}\right)\gamma(a^{-1},\psi)^{\ell-1}
	\end{eqnarray*}
	
	In summary, we get 
	\begin{equation*}\label{eq2}
	I(a,r)=|a|^{\frac{1}{2}\left[\frac{r-1}{2}\right]+1}\psi\left(\frac{r}{2a}\right)\left(\frac{r}{2^{r-1}},a^{-1}\right)\gamma\left(a^{-1},\psi\right)^{\left[\frac{r-1}{2}\right]}.
	\end{equation*}
	
	
	\end{proof}
Combining Equation (\ref{i1}), Equation (\ref{i2}), Propostions \ref{charac2} and \ref{prop1} we obtain the following proposition:
\begin{prop}\label{pro2}
	Suppose that the residual characteristic of $F$ is larger than $2r+1$. For $$\alpha=\D(a,a,\dots,a^{1-r}\det(w_M)\det(w_{\G_r})$$
	and $|a|$ is small enough,
	$$L_{w_M}^{w_{\G_r}}(\alpha)=|a|^{\left[\frac{r^2}{4}\right]-\frac{1}{2}\left[\frac{r}{2}\right]}\gamma(a^{-1},\psi)^{-\left[\frac{r}{2}\right]}K_{w_M}^{w_{\G_r}}(\alpha).$$
\end{prop}
Note that to simplify the formula, we used the following identities:
$$-\left[\frac{r^2}{4}\right]+\left[\frac{r^2+2r-3}{4}\right]-\left[\frac{r+1}{2}\right]+1=0$$
and $$(r-2)-\left[\frac{r^2+2r-3}{4}\right]+\frac{1}{2}\left[\frac{r-1}{2}\right]+1+\frac{(r-1)^2}{2}=\left[\frac{r^2}{4}\right]-\frac{1}{2}\left[\frac{r}{2}\right].$$ 
\section{Proof of the main theorem}
We shall prove the Theorem \ref{thm1} by induction on $r$. It is trivial when $r=1$. We suppose that it holds for $1\leq r'<r$.

Firstly we consider $W^R_r \ni w\neq \,w_{\G_r}$. The relevant element $w\mathbf{t}$ then has the following form
$$w\mathbf{t}=\left(\begin{matrix}
w_1\mathbf{t}_1&0\\
0&w_2\mathbf{t}_2
\end{matrix}\right)$$
with $w_i\mathbf{t}_i$ is the relevant element of $\G_{r_i}$. For a convenience, we shall introduce some intermediate orbital integrals.

On the side $J$, for a function $f\in\mathcal{C}^{\infty}_c(\G_{r_1+r_2}(F))$ 
we define the intermediate integral 
\begin{multline*}
J^{r_1}_{r_2}\left[\left(\begin{matrix}
g_1&\\
&g_2
\end{matrix}\right),f\right]\\
:=\int f\left[w_{\G_{r_1+r_2}}\left(\begin{matrix}
\mathrm{Id}_{r_1}&\\
{}^tX&\mathrm{Id}_{r_2}
\end{matrix}\right)\left(\begin{matrix}
w_{\G_{r_1}}g_1&\\
&w_{\G_{r_2}}g_2
\end{matrix}\right)\left(\begin{matrix}
\mathrm{Id}_{r_1}&Y\\
&\mathrm{Id}_{r_2}
\end{matrix}\right)\right]\times\\\times\theta\left[\left(\begin{matrix}
\mathrm{Id}_{r_1}&X+Y\\
&\mathrm{Id}_{r_2}
\end{matrix}\right)\right]dXdY\\
=\int f\left[w_{\G_{r_1+r_2}}\left(\begin{matrix}
A_{r_1}&A_{r_1}Y\\
{}^tXA_{r_1}&{}^tXA_{r_1}Y+B_{r_2}
\end{matrix}\right)\right]\theta\left[\left(\begin{matrix}
\mathrm{Id}_{r_1}&X+Y\\
&\mathrm{Id}_{r_2}
\end{matrix}\right)\right]dXdY,
\end{multline*}
where $A_{r_1}:=w_{\G_{r_1}}g_1\in \G_{r_1}(F)$, $B_{r_2}:=w_{\G_{r_2}}g_2\in \G_{r_2}(F)$ and the domain of integration is $M_{r_1\times r_2}(F)$- the set of matrices of size $r_1\times r_2$. 

Fixing $f\in\mathcal{C}^{\infty}_c(\G_{r_1+r_2}(F))$, the function $\mathfrak{J}^{r_1}_{r_2}(g_1,g_2):=J^{r_1}_{r_2}\left[\left(\begin{matrix}
g_1&\\
&g_2
\end{matrix}\right),f\right]$ is a smooth function of support compact on $\G_{r_1}(F)\times \G_{r_2}(F)$. 

Associated with the action of $(N_{r_1}\times N_{r_1})\times (N_{r_2}\times N_{r_2})$ on $\G_{r_1}\times \G_{r_2}$ we consider the double orbital integral 
\begin{multline*}
J(w_1\mathbf{t}_1,w_2\mathbf{t}_2;\mathfrak{J}^{r_1}_{r_2})\\:=\int\int \mathfrak{J}^{r_1}_{r_2}(n_1^{-1}w_{\G_{r_1}}w_1\mathbf{t}_1n'_1,n_2^{-1}w_{\G_{r_2}}w_2\mathbf{t}_2n'_2)\theta(n_1n'_1)
dn_1dn'_1\theta(n_2n'_2)dn_2dn'_2,
\end{multline*}
where the $(n_i,n'_i)$ is integrated over $N_{r_i}\times N_{r_i}$ divided by the stabilizer of $w_{\G_{r_i}}w_i\mathbf{t}_i$. 
We have then
\begin{equation}\label{eqn4}J\left(\left(\begin{matrix}
w_1\mathbf{t}_1&\\
&w_2\mathbf{t}_2
\end{matrix}\right),f\right)=J(w_1\mathbf{t}_1,w_2\mathbf{t}_2;\mathfrak{J}^{r_1}_{r_2})
\end{equation}

We can also define the partial orbital integrals $J_1(w_1\mathbf{t}_1,g_2;\mathfrak{J}^{r_1}_{r_2})$ and  $J_2(g_1,w_2\mathbf{t}_2;\mathfrak{J}^{r_1}_{r_2})$. For example 
$$
J_1(w_1\mathbf{t}_1,g_2;\mathfrak{J}^{r_1}_{r_2}):=\int \mathfrak{J}^{r_1}_{r_2}(n_1^{-1}w_{\G_{r_1}}w_1\mathbf{t}_1n'_1,g_2)\theta(n_1n'_1)
dn_1dn'_1,
$$
where the $(n_1,n'_1)$ is integrated over $N_{r_1}\times N_{r_1}$ divided by the stabilizer of $w_{\G_{r_1}}w_1\mathbf{t}_1$. If we fix $w_1\mathbf{t}_1$, this integral defines a smooth function of compact support on $\G_{r_2}(F)$. Moreover, we have: 
$$J(w_1\mathbf{t}_1,w_2\mathbf{t}_2;\mathfrak{J}^{r_1}_{r_2})=J(w_2\mathbf{t}_2,J_1(w_1\mathbf{t}_1,\bullet;\mathfrak{J}^{r_1}_{r_2}))$$
and $$J(w_1\mathbf{t}_1,w_2\mathbf{t}_2;\mathfrak{J}^{r_1}_{r_2})=J(w_1\mathbf{t}_1,J_2(\bullet,                                                                                                                                                                                                                                                                                                                                                                                                                                                                                                                                                                                                                                                                                                                                                                                                                                                                                                                                                                                                                                                                                                                                                                                                                                                                                                                                              w_2\mathbf{t}_2;\mathfrak{J}^{r_1}_{r_2})).$$

On the side $I$, for a function $\phi\in\mathcal{C}^{\infty}_c(\G_{r_1+r_2}(F))$ 
we define the intermediate integral 
\begin{multline*}
I^{r_1}_{r_2}\left[\left(\begin{matrix}
g_1&\\
&g_2
\end{matrix}\right),\phi\right]\\
:=\int \phi\left[\left(\begin{matrix}
\mathrm{Id}_{r_1}&\\
{}^tX&\mathrm{Id}_{r_2}
\end{matrix}\right)\left(\begin{matrix}
g_1&\\
&g_2
\end{matrix}\right)\left(\begin{matrix}
\mathrm{Id}_{r_1}&X\\
&\mathrm{Id}_{r_2}
\end{matrix}\right)\right]\theta\left[\left(\begin{matrix}
\mathrm{Id}_{r_1}&2X\\
&\mathrm{Id}_{r_2}
\end{matrix}\right)\right]dX\\
=\int \phi\left[\left(\begin{matrix}
g_1&g_1X\\
{}^tXg_1&{}^tXg_1X+g_2
\end{matrix}\right)\right]\theta\left[\left(\begin{matrix}
\mathrm{Id}_{r_1}&2X\\
&\mathrm{Id}_{r_2}
\end{matrix}\right)\right]dX,
\end{multline*}
where $g_i\in \G_{r_i}(F)$ and the domain of integration is $M_{r_1\times r_2}(F)$- the set of matrices of size $r_1\times r_2$. 

Fixing $\phi\in\mathcal{C}^{\infty}_c(\G_{r_1+r_2}(F))$, the function $\mathfrak{I}^{r_1}_{r_2}(g_1,g_2):=I^{r_1}_{r_2}\left[\left(\begin{matrix}
g_1&\\
&g_2
\end{matrix}\right),\phi\right]$ is a smooth function of support compact on $\G_{r_1}(F)\times \G_{r_2}(F)$. 

Associated with the action of $N_{r_1}\times N_{r_2}$ on $\G_{r_1}\times \G_{r_2}$ we consider the double orbital integral 
$$
I(w_1\mathbf{t}_1,w_2\mathbf{t}_2;\mathfrak{I}^{r_1}_{r_2}):=\int \mathfrak{I}^{r_1}_{r_2}({}^tn_1w_1\mathbf{t}_1n_1,{}^tn_2w_2\mathbf{t}_2n_2)\theta_{r_1}(n_1)
dn_1\theta_{r_2}(n_2)dn_2,
$$
where the $n_i$ is integrated over $N_{r_i}$ divided by the stabilizer of $w_i\mathbf{t}_i$. 
We have then
\begin{equation}\label{eqn5}I\left(\left(\begin{matrix}
w_1\mathbf{t}_1&\\
&w_2\mathbf{t}_2
\end{matrix}\right),\phi\right)=I(w_1\mathbf{t}_1,w_2\mathbf{t}_2;\mathfrak{I}^{r_1}_{r_2})
\end{equation}

We can also define the partial orbital integrals $I_1(w_1\mathbf{t}_1,g_2;\mathfrak{I}^{r_1}_{r_2})$ and  $I_2(g_1,w_2\mathbf{t}_2;\mathfrak{I}^{r_1}_{r_2})$. For example 
$$
I_1(w_1\mathbf{t}_1,g_2;\mathfrak{I}^{r_1}_{r_2}):=\int \mathfrak{I}^{r_1}_{r_2}({}^tn_1w_1\mathbf{t}_1n_1,g_2)\theta_{r_1}(n_1)
dn_1,
$$
where the $n_1$ is integrated over $N_{r_1}$ divided by the stabilizer of $w_1\mathbf{t}_1$. If we fix $w_1\mathbf{t}_1$, this integral defines a smooth function of compact support on $\G_{r_2}(F)$. Moreover, we have: 
$$I(w_1\mathbf{t}_1,w_2\mathbf{t}_2;\mathfrak{I}^{r_1}_{r_2})=I(w_2\mathbf{t}_2,I_1(w_1\mathbf{t}_1,\bullet;\mathfrak{I}^{r_1}_{r_2}))$$
and $$I(w_1\mathbf{t}_1,w_2\mathbf{t}_2;\mathfrak{I}^{r_1}_{r_2})=I(w_1\mathbf{t}_1,I_2(\bullet,                                                                                                                                                                                                                                                                                                                                                                                                                                                                                                                                                                                                                                                                                                                                                                                                                                                                                                                                                                                                                                                                                                                                                                                                                                                                                                                                              w_2\mathbf{t}_2;\mathfrak{I}^{r_1}_{r_2})).$$

Now we can continue with the induction argument. If $w_i=\mathrm{Id}_{r_i}$ then $w=\mathrm{Id}_r$. Suppose that $J(\D(\mathbf{t_1},\mathbf{t_2}),f)=\Delta_r(\D(\mathbf{t}_1,\mathbf{t}_2))I(\D(\mathbf{t_1},\mathbf{t_2}),\phi)$. Using the identities (\ref{eqn4}) and (\ref{eqn5}) we have: 
$$J(\mathbf{t}_1,\mathbf{t}_2;\mathfrak{J}^{r_1}_{r_2})=
\Delta_r(\D(\mathbf{t}_1,\mathbf{t}_2))I(\mathbf{t}_1,\mathbf{t}_2;\mathfrak{I}^{r_1}_{r_2})$$
Using the relation between the double integrals and the partial integral, it implies 
\begin{eqnarray}\label{eq6}J(\mathbf{t}_1,J_2(\bullet,\mathbf{t}_2;\mathfrak{J}^{r_1}_{r_2}))&=&
\Delta_r(\D(\mathbf{t}_1,\mathbf{t}_2))I(\mathbf{t}_1,I_2(\bullet,\mathbf{t}_2;\mathfrak{I}^{r_1}_{r_2}))
\end{eqnarray}
\begin{itemize}
	\item If $r\equiv r_1\mod 2$ (it is equivalent to $r_2\equiv 0\mod 2$), we have then: 
	$$\frac{\Delta_r(\D(\mathbf{t}_1,\mathbf{t}_2))}{\Delta_{r_1}(\mathbf{t}_1)}=\zeta(-1)^{\frac{r_2}{2}v(\det(\mathbf{t}_1))}|\det(\mathbf{t}_1)|^{-\frac{r_2}{2}}\Delta_{r_2}(\mathbf{t}_2).$$
	For fixed $\mathbf{t}_2\in T_{r_2}(F)$ we define 
	$$I'_2(g_1,\mathbf{\mathbf{t}_2};\mathfrak{I}^{r_1}_{r_2})=\zeta(-1)^{\frac{r_2}{2}v(\det(g_1))}|\det(g_1)|^{-\frac{r_2}{2}}\Delta_{r_2}(\mathbf{t}_2)I_2(g_1,\mathbf{\mathbf{t}_2};\mathfrak{I}^{r_1}_{r_2}).$$ 
	It is a smooth function of compact support on $\mathrm{GL}_{r_1}(F)$.
	
	Since the identity (\ref{eq6}) is true for all $\mathbf{t}_i\in T_{r_i}$, so when we fix $\mathbf{t}_2\in T_{r_2}(F)$, we obtain the matching relation over $\G_{r_1}(F)$:
	$$J(\mathbf{t}_1,J_2(\bullet,\mathbf{t}_2;\mathfrak{J}^{r_1}_{r_2}))=
	\Delta_{r_1}(\mathbf{t}_1)I(\mathbf{t}_1,I'_2(\bullet,\mathbf{t}_2;\mathfrak{I}^{r_1}_{r_2}))\quad\forall \,\mathbf{t_1}\in T_{r_1}(F).
	$$
	By induction, there exists $\Delta_{w_1}$ such that 
	\begin{align*}
	&J(w_1\mathbf{t}_1,J_2(\bullet,\mathbf{t}_2;\mathfrak{J}^{r_1}_{r_2}))=\Delta_{w_1}(\mathbf{t}_1)I(w_1\mathbf{t}_1,I'_2(\bullet,\mathbf{t}_2;\mathfrak{I}^{r_1}_{r_2}))\quad \forall \,\mathbf{t_1}\in T_{w_1}(F)\\
	&=\Delta_{w_1}(\mathbf{t}_1)\zeta(-1)^{\frac{r_2}{2}v(\det(w_1\mathbf{t}_1))}|\det(w_1\mathbf{t}_1)|^{-\frac{r_2}{2}}\Delta_{r_2}(\mathbf{t}_2)I(w_1\mathbf{t}_1,I_2(\bullet,\mathbf{t}_2;\mathfrak{I}^{r_1}_{r_2}))\\
	&=\Delta_{r_2}(\mathbf{t}_2)\zeta(-1)^{\frac{r_2}{2}v(\det(\mathbf{t}_1))}|\det(\mathbf{t}_1)|^{-\frac{r_2}{2}}\Delta_{w_1}(\mathbf{t}_1)I(w_1\mathbf{t}_1,I_2(\bullet,\mathbf{t}_2;\mathfrak{I}^{r_1}_{r_2})).	
	\end{align*}
	
	Reusing the relation between the double integrals and the partial integrals, we obtain the matching relation over $\G_{r_2}(F)$
	$$J(\mathbf{\mathbf{t}_2},J_1(w_1\mathbf{t}_1,\bullet;\mathfrak{J}^{r_1}_{r_2}))=\Delta_{r_2}(\mathbf{t}_2)I(\mathbf{t}_2,I'_1(w_1\mathbf{t}_1,\bullet;\mathfrak{I}^{r_1}_{r_2}))\quad \forall\, \mathbf{t_2}\in T_{r_2}(F),$$
	where $$I'_1(w_1\mathbf{t}_1,g_2;\mathfrak{I}^{r_1}_{r_2})=\zeta(-1)^{\frac{r_2}{2}v(\det(\mathbf{t}_1))}|\det(\mathbf{t}_1)|^{-\frac{r_2}{2}}\Delta_{w_1}(\mathbf{t}_1)I_1(w_1\mathbf{t}_1,g_2;\mathfrak{I}^{r_1}_{r_2}).$$ We should note that for fixed $w_1\mathbf{\mathbf{t}_1}$, the function $I'_1(w_1\mathbf{t}_1,g_2;\mathfrak{I}^{r_1}_{r_2})$ is a smooth function of compact support over $\G_{r_2}(F)$.
	
	By induction, there exists $\Delta_{w_2}$ such that  
	\begin{align*}
	&J(w_2\mathbf{\mathbf{t}_2},J_1(w_1\mathbf{t}_1,\bullet;\mathfrak{J}^{r_1}_{r_2}))=\Delta_{w_2}(\mathbf{t}_2)I(w_2\mathbf{t}_2,I'_1(w_1\mathbf{t}_1,\bullet;\mathfrak{I}^{r_1}_{r_2}))\quad\forall \,\mathbf{t_2}\in T_{w_2}(F)\\
	&=\zeta(-1)^{\frac{r_2}{2}v(\det(\mathbf{t}_1))}|\det(\mathbf{t}_1)|^{-\frac{r_2}{2}}\Delta_{w_1}(\mathbf{t}_1) \Delta_{w_2}(\mathbf{t}_2)I(w_2\mathbf{t_2},I_1(w_1\mathbf{t}_1,\bullet;\mathfrak{I}^{r_1}_{r_2})).
	\end{align*}
	
	We have then $\Delta_{w}(\mathbf{t})=\zeta(-1)^{\frac{r_2}{2}v(\det(\mathbf{t}_1))}|\det(\mathbf{t}_1)|^{-\frac{r_2}{2}}\Delta_{w_1}(\mathbf{t}_1) \Delta_{w_2}(\mathbf{t}_2).$
	\item If $r\not \equiv r_1\mod 2$ (it is equivalent to $r_2\equiv 1\mod 2$), we have then: 
	$$\frac{\Delta_r(\D(\mathbf{t}_1,\mathbf{t}_2))}{\Delta'_{r_1}(\mathbf{t}_1)}=\zeta(-1)^{\frac{r_2-1}{2}v(\det(\mathbf{t}_1))}|\det(\mathbf{t}_1)|^{-\frac{r_2}{2}}\Delta_{r_2}(\mathbf{t}_2).$$
	By doing the same (as the above case), we obtain that:
	$$\Delta_{w}(\mathbf{t})=\zeta(-1)^{\frac{r_2-1}{2}v(\det(\mathbf{t}_1))}|\det(\mathbf{t}_1)|^{-\frac{r_2}{2}}\Delta'_{w_1}(\mathbf{t}_1)\Delta_{w_2}(\mathbf{t}_2) .$$
	
\end{itemize}

We have proved the following proposition:
\begin{prop}\label{redu1}
	Let $w\mathbf{t}=\left(\begin{matrix}
	w_1\mathbf{t}_1&0\\
	0&w_2\mathbf{t}_2
	\end{matrix}\right)$ be a relevant element of $\G_r$. We have then
	$$\Delta_{w}(\mathbf{t})=\begin{cases}
\zeta(-1)^{\frac{r_2}{2}v(\det(\mathbf{t}_1))}|\det(\mathbf{t}_1)|^{-\frac{r_2}{2}}\Delta_{w_1}(\mathbf{t}_1) \Delta_{w_2}(\mathbf{t}_2)&\text{if } r_2\equiv 0\mod 2\\
\zeta(-1)^{\frac{r_2-1}{2}v(\det(\mathbf{t}_1))}|\det(\mathbf{t}_1)|^{-\frac{r_2}{2}}\Delta'_{w_1}(\mathbf{t}_1)\Delta_{w_2}(\mathbf{t}_2)&\text{if } r_2\not\equiv 0\mod 2.
	\end{cases}$$
\end{prop}

Similarly, we can also obtain the following proposition:

\begin{prop}\label{redu2}
	Let $w\mathbf{t}=\left(\begin{matrix}
	w_1\mathbf{t}_1&0\\
	0&w_2\mathbf{t}_2
	\end{matrix}\right)$ be a relevant element of $\G_r$. We have then
	$$\Delta'_{w}(\mathbf{t})=\begin{cases}
	\zeta(-1)^{\frac{r_2}{2}v(\det(\mathbf{t}_1))}|\det(\mathbf{t}_1)|^{-\frac{r_2}{2}}\Delta'_{w_1}(\mathbf{t}_1) \Delta'_{w_2}(\mathbf{t}_2)&\text{if } r_2\equiv 0\mod 2\\
	\zeta(-1)^{\frac{r_2+1}{2}v(\det(\mathbf{t}_1))}|\det(\mathbf{t}_1)|^{-\frac{r_2}{2}}\Delta_{w_1}(\mathbf{t}_1)\Delta'_{w_2}(\mathbf{t}_2)&\text{if } r_2\not\equiv 0\mod 2.
	\end{cases}$$
\end{prop}
\begin{coro}\label{mirabolic}
	Let $M$ be the standard parabolic subgroup type $(r-1,1)$ of $\G_r$. Let $\mathbf{t}=\D(a,\dots,a,b)\in T_{w_M}$ We have then:
	$$\Delta_{w_M}(\mathbf{t})=|a|^{-\frac{r-1}{2}}\Delta'_{w_{\G_{r-1}}}(\D(a,\dots,a))$$ and
	$$\Delta'_{w_M}(\mathbf{t})=\zeta(-1)^{v(\det(\mathbf{t}))}|a|^{-\frac{r-1}{2}}\gamma(b,\psi)(b,\varpi)^{v(b)}\Delta_{w_{\G_{r-1}}}(\D(a,\dots,a)).$$
\end{coro}



The rest of relevant elements are of the type $\Delta_{w_{\G_r}}\mathbf{t}$. For working with this type of relevant element, we shall use the germ relations which are mentioned in the section 2 and 3.  

Recall that we have:
$$J(w_M\mathbf{t},f)=\omega_f(\mathbf{t})+\sum_{\alpha\beta=\mathbf{t}}K^{w_{\G_r}}_{w_M}(\alpha)J(w_{\G_r}\beta,f)$$
and 
$$I(w_M\mathbf{t},\phi)=\omega_\phi(\mathbf{t})+\sum_{\alpha\beta=\mathbf{t}}L^{w_{\G_r}}_{w_M}(\alpha)I(w_{\G_r}\beta,\phi),$$
where $w_f,w_\phi$ are smooth function of compact support, $\mathbf{t}\in T_{w_M}$ and the sums are over all pairs in $\mathfrak{S}:=\{(\alpha,\beta)\in (T_{w_M}^{w_{\G_r}},T_{w_{\G_r}})|\alpha\beta=\mathbf{t}\}$. Given $(\alpha,\beta)\in (T_{w_M}^{w_{\G_r}},T_{w_{\G_r}})$, then all the pair $(\alpha',\beta')\in (T_{w_M}^{w_{\G_r}},T_{w_{\G_r}})$ satisfied $\alpha'\beta'=\alpha\beta$ have a form $(z^{-1}\alpha,z\beta)$ with $z$ is a $r$-th root of unity.

Given $\beta\in T_{w_{\G_r}}$, we choose $\alpha=\D(a,\dots,a,a^{1-r}\det(w_Mw_{\G_r}))$ with $|a|$ so small that $\omega_f(\alpha\beta) =\omega_\phi(\alpha\beta)=0$. We get then:
$$J(w_M\alpha\beta,f)=\sum_{z|z^r=1}K^{w_{\G_r}}_{w_M}(z^{-1}\alpha)J(w_{\G_r}z\beta,f)$$
and 
$$I(w_M\alpha\beta,\phi)=\sum_{z|z^r=1}L^{w_{\G_r}}_{w_M}(z^{-1}\alpha)I(w_{\G_r}z\beta,\phi).$$
Moreover, with $|a|$ small enough, using Proposition \ref{pro2}, we have: 
$$L^{w_{\G_r}}_{w_M}(z^{-1}\alpha)=c(a,z).K^{w_{\G_r}}_{w_M}(z^{-1}\alpha)$$
with $c(a,z):=|az^{-1}|^{\left[\frac{r^2}{4}\right]-\frac{1}{2}\left[\frac{r}{2}\right]}\gamma((az^{-1})^{-1},\psi)^{-\left[\frac{r}{2}\right]}$. Combining them with the matching relation of $\phi$ and $f$
 on the orbit of $w_M\alpha\beta$, we obtain
\begin{equation}\label{important}
\sum_{z|z^r=1}K^{w_{\G_r}}_{w_M}(z^{-1}\alpha)(J(w_{\G_r}z\beta,f)-c(a,z).\Delta^{\bullet}_{w_M}(\alpha\beta)I(w_{\G_r}z\beta,\phi))=0,
\end{equation}
where $\Delta^{\bullet}_{w_M}(\alpha\beta)$ is $\Delta_{w_M}(\alpha\beta)$ or $\Delta'_{w_M}(\alpha\beta)$ depend on the matching relation.

We hope that this condition implies 
$$J(w_{\G_r}z\beta,f)-c(a,z).\Delta^\bullet_{w_M}(\alpha\beta)I(w_{\G_r}z\beta,\phi)=0$$ for all $z$ satisfied $z^r=1$. In particularly, we have: 
$$J(w_{\G_r}\beta,f)=c(a,1).\Delta^\bullet_{w_M}(\alpha\beta)I(w_{\G_r}\beta,\phi).$$



For the fixed $\beta$, $f$ and $\phi$, the two integrals $I$ and $J$ don't depend on the choice of $a$, so we can require that $v(a)$ is even. With this addition condition, we define that:
\begin{equation}\label{smallest 2}
\Delta^\bullet_{w_{\G_r}}(\beta):=c(a,1).\Delta^\bullet_{w_M}(\alpha\beta).
\end{equation}

Now we shall prove that this definition is well defined (i.e the definition doesn't depend on $a$). More precisely, we prove the following proposition:
\begin{prop}
	Let $\beta\in F^\times$. Viewing $\beta$ as an element of $T_{w_{\G_r}}(F)$, we have then:
	$$\Delta_{w_{\G_r}}(\beta)=\zeta(-1)^{\left[\frac{r}{2}\right]\left[\frac{r+1}{2}\right]v(\beta)}|\beta|^{\frac{1}{2}\left[\frac{r}{2}\right]-\left[\frac{r^2}{4}\right]}\gamma(\beta,\psi)^{\left[\frac{r}{2}\right]}(\beta,\varpi)^{\left[\frac{r}{2}\right]v(\beta)}$$ and 
	$$\Delta'_{w_{\G_r}}(\beta)=\zeta(-1)^{\left[\frac{r+1}{2}\right]\left[\frac{r+2}{2}\right]v(\beta)}|\beta|^{\frac{1}{2}\left[\frac{r}{2}\right]-\left[\frac{r^2}{4}\right]}\gamma(\beta,\psi)^{\left[\frac{r+1}{2}\right]}(\beta,\varpi)^{\left[\frac{r+1}{2}\right]{v(\beta)}}.$$
\end{prop} 
\begin{proof}We shall prove this proposition by induction on $r$.
	\begin{itemize}
		\item For $r=1$, it is trivial. In this case, we have $$\Delta_{w_{\G_1}}(\beta)=\Delta_1(\beta)=1$$ and $$\Delta'_{w_{\G_1}}(\beta)=\Delta'_1(\beta)=\zeta(-1)^{v(\beta)}\gamma(\beta,\psi)(\beta,\varpi)^{v(\beta)}.$$
		\item Assume that this proposition holds for $r-1$. Combining the equation (\ref{smallest 2}) and the formula in Corollary \ref{mirabolic} we have:
		\begin{eqnarray}\label{important3}
		\Delta_{w_{\G_r}}(\beta)&=&|a|^{\left[\frac{r^2}{4}\right]-\frac{1}{2}\left[\frac{r}{2}\right]}\gamma(a^{-1},\psi)^{-\left[\frac{r}{2}\right]}\Delta_M(\alpha\beta)\nonumber\\
		&=&|a|^{\left[\frac{r^2}{4}\right]-\frac{1}{2}\left[\frac{r}{2}\right]}\Delta_M(\alpha\beta)\quad\left(\text{since $v(a)$ is even then $\gamma(a^{-1},\psi)=1$}\right)\nonumber\\
		&=&|a|^{\left[\frac{r^2}{4}\right]-\frac{1}{2}\left[\frac{r}{2}\right]}|a\beta|^{-\frac{r-1}{2}}\Delta'_{w_{\G_{r-1}}}(\D(a\beta,\dots,a\beta))\nonumber\\
		&=&|a|^{\left[\frac{r^2}{4}\right]-\frac{1}{2}\left[\frac{r}{2}\right]}|a\beta|^{-\frac{r-1}{2}}\nonumber\\
		&&\times\zeta(-1)^{\left[\frac{r}{2}\right]\left[\frac{r+1}{2}\right]v(a\beta)}|a\beta|^{\frac{1}{2}\left[\frac{r-1}{2}\right]-\left[\frac{(r-1)^2}{4}\right]}\gamma(a\beta,\psi)^{\left[\frac{r}{2}\right]}(a\beta,\varpi)^{\left[\frac{r}{2}\right]v(a\beta)}\nonumber\\
		&=&\zeta(-1)^{\left[\frac{r}{2}\right]\left[\frac{r+1}{2}\right]v(\beta)}|\beta|^{\frac{1}{2}\left[\frac{r}{2}\right]-\left[\frac{r^2}{4}\right]}\gamma(a\beta,\psi)^{\left[\frac{r}{2}\right]}(a\beta,\varpi)^{\left[\frac{r}{2}\right]{v(\beta)}}.
		\end{eqnarray}
		The last equation is obtained by using $v(a)\equiv 0\mod 2$ and the identity:
		$$\left[\frac{r^2}{4}\right]-\frac{1}{2}\left[\frac{r}{2}\right]=\left[\frac{(r-1)^2}{4}\right]-\frac{1}{2}\left[\frac{r-1}{2}\right]+\frac{r-1}{2}.$$ 
		
	Moreover, we have:
	\begin{eqnarray}
	\gamma(a\beta,\psi)(a\beta,\varpi)^{v(\beta)}&=&\gamma(a,\psi)\gamma(\beta,\psi)(a,\beta)(a\beta,\varpi)^{v(\beta)}\nonumber\\
	&=&\gamma(\beta,\psi)(a,\beta)(a\beta,\varpi)^{v(\beta)}\,(\text{since $a$ is even then $\gamma(a,\psi)=1$})\nonumber\\
	&=&\gamma(\beta,\psi)(a,\varpi)^{v(\beta)}(a\beta,\varpi)^{v(\beta)}\label{important2}\\
	&=&\gamma(\beta,\psi)(\beta,\varpi)^{v(\beta)}.\label{important4}
	\end{eqnarray}	Since $v(a)$ is even, we have $(a,\beta)=\begin{cases}
	1&\text{if $v(\beta)$ is even} \\
	(a,\varpi)&\text{if $v(\beta)$ is odd}
	\end{cases}=(a,\varpi)^{v(\beta)}$. It implies the relation (\ref{important2}).
	
	Combining the relations (\ref{important3}) and (\ref{important4}), we get:
	$$\Delta_{w_{\G_r}}(\beta)=\zeta(-1)^{\left[\frac{r}{2}\right]\left[\frac{r+1}{2}\right]v(\beta)}|\beta|^{\frac{1}{2}\left[\frac{r}{2}\right]-\left[\frac{r^2}{4}\right]}\gamma(\beta,\psi)^{\left[\frac{r}{2}\right]}(\beta,\varpi)^{\left[\frac{r}{2}\right]v(\beta)}$$
	
		Similarly, we have: 
	\begin{eqnarray*}
		\Delta'_{w_{\G_r}}(\beta)&=&|a|^{\left[\frac{r^2}{4}\right]-\frac{1}{2}\left[\frac{r}{2}\right]}\gamma(a^{-1},\psi)^{-\left[\frac{r}{2}\right]}\Delta'_M(\alpha\beta)\\
		&=&|a|^{\left[\frac{r^2}{4}\right]-\frac{1}{2}\left[\frac{r}{2}\right]}\Delta'_M(\alpha\beta)\quad\left(\text{since $a$ is square then $\gamma(a^{-1},\psi)=1$}\right)\\
		&=&|a|^{\left[\frac{r^2}{4}\right]-\frac{1}{2}\left[\frac{r}{2}\right]}\zeta(-1)^{v((-1)^{r-1}\beta^r)}|a\beta|^{-\frac{r-1}{2}}\gamma((-a)^{1-r}\beta,\psi)\\
		&&\times((-a)^{1-r}\beta,\varpi)^{v((-a)^{1-r}\beta)}\Delta_{w_{\G_{r-1}}}(a\beta)\\
		&=&|a|^{\left[\frac{r^2}{4}\right]-\frac{1}{2}\left[\frac{r}{2}\right]}\zeta(-1)^{v(\beta^r)}|a\beta|^{-\frac{r-1}{2}}\gamma((-a)^{1-r}\beta,\psi)((-a)^{1-r}\beta,\varpi)^{v(\beta)}\\
		&&\times\zeta(-1)^{\left[\frac{r-1}{2}\right]\left[\frac{r}{2}\right]v(a\beta)}|a\beta|^{\frac{1}{2}\left[\frac{r-1}{2}\right]-\left[\frac{(r-1)^2}{4}\right]}\gamma(a\beta,\psi)^{\left[\frac{r-1}{2}\right]}(a\beta,\psi)^{\left[\frac{r-1}{2}\right]v(a\beta)}\\
		&=&\zeta(-1)^{\left[\frac{r+1}{2}\right]\left[\frac{r+2}{2}\right]v(\beta)}|\beta|^{\frac{1}{2}\left[\frac{r}{2}\right]-\left[\frac{r^2}{4}\right]}\gamma(\beta,\psi)^{\left[\frac{r+1}{2}\right]}(\beta,\varpi)^{\left[\frac{r+1}{2}\right]}.
	\end{eqnarray*}
	The last equation is obtained by using the relation (\ref{important4}) (note that it requires only $v(a)\equiv 0\mod 2$) and the identity:
	$$\left[\frac{r-1}{2}\right]\left[\frac{r}{2}\right]+r=\left[\frac{r+1}{2}\right]\left[\frac{r+2}{2}\right].$$ 
	
	\end{itemize}
\end{proof}
\begin{lemm}\label{technique}Suppose that $v(a)$ is even. We have 
	$$\Delta_{w_{\G_r}}(z\beta)=c(a,z)\Delta_{w_M}(\alpha\beta)$$
	and 
	$$\Delta'_{w_{\G_r}}(z^{-1}\beta)=c(a,z)\Delta'_{w_M}(\alpha\beta)$$ for all $z\in F^*$ satisfied $z^r=1$ and $\beta\in F^*$.
\end{lemm}
\begin{proof}
	We prove only the first relation (the second one can be done similarly). We should mention that we use $v(z)=0$ in some steps of the following calculation.
	\begin{eqnarray*}
		c(a,z)\Delta_{w_M}(\alpha\beta)&=&|az^{-1}|^{\left[\frac{r^2}{4}\right]-\frac{1}{2}\left[\frac{r}{2}\right]}\gamma((az^{-1})^{-1},\psi)^{-\left[\frac{r}{2}\right]}\Delta_M(\alpha\beta)\nonumber\\
		&=&|az^{-1}|^{\left[\frac{r^2}{4}\right]-\frac{1}{2}\left[\frac{r}{2}\right]}\Delta_M(\alpha\beta)\quad\left(\text{since $v(az^{-1})$ is even}\right)\nonumber\\ 
		&=&|az^{-1}|^{\left[\frac{r^2}{4}\right]-\frac{1}{2}\left[\frac{r}{2}\right]}|a\beta|^{-\frac{r-1}{2}}\Delta'_{w_{\G_{r-1}}}(\D(a\beta,\dots,a\beta))\nonumber\\
		&=&|az^{-1}|^{\left[\frac{r^2}{4}\right]-\frac{1}{2}\left[\frac{r}{2}\right]}|a\beta|^{-\frac{r-1}{2}}\times\nonumber\\
		&&\times\zeta(-1)^{\left[\frac{r}{2}\right]\left[\frac{r+1}{2}\right]v(a\beta)}|a\beta|^{\frac{1}{2}\left[\frac{r-1}{2}\right]-\left[\frac{(r-1)^2}{4}\right]}\gamma(a\beta,\psi)^{\left[\frac{r}{2}\right]}(a\beta,\psi)^{\left[\frac{r}{2}\right]v(a\beta)}\nonumber\\
		&=&\zeta(-1)^{\left[\frac{r}{2}\right]\left[\frac{r+1}{2}\right]v(\beta)}|z\beta|^{\frac{1}{2}\left[\frac{r}{2}\right]-\left[\frac{r^2}{4}\right]}\gamma(\beta,\psi)^{\left[\frac{r}{2}\right]}(\beta,\varpi)^{v(\beta)}\nonumber\\
		&=&\zeta(-1)^{\left[\frac{r}{2}\right]\left[\frac{r+1}{2}\right]v(z\beta)}|z\beta|^{\frac{1}{2}\left[\frac{r}{2}\right]-\left[\frac{r^2}{4}\right]}\gamma(z\beta,\psi)^{\left[\frac{r}{2}\right]}(z\beta,\varpi)^{v(\beta)} \\
		&=&\Delta_{w_{\G_r}}(z\beta).
	\end{eqnarray*}	

\end{proof}
The relation (\ref{important}) reads
$$\sum_{z|z^r=1}K^{w_{\G_r}}_{w_M}(z^{-1}\alpha)(J(w_{\G_r}z\beta,f)-\Delta^\bullet_{w_{\G_r}}(z\beta)I(w_{\G_r}z\beta,\phi))=0,
$$ 
for $|a|$ small enough and $v(a)\equiv 0\mod 2$. If we set 
$$m^\bullet(z)=J(w_{\G_r}z\beta,f)-\Delta^\bullet_{w_{\G_r}}(z\beta)I(w_{\G_r}z\beta,\phi)$$ we see that the above relation reads:
$$\sum_{z|z^r=1}\psi(\frac{rz}{2a})m^\bullet(z)=0,$$
for $|a|$ small enough and $v(a)\equiv 0\mod 2$. We have to see that $m^\bullet(z)=0$ for all $z$. Thus it will follows from the following lemma.
\begin{lemm} Given $x_i$ are distinct point in $F$. Suppose that for each index $i$, there is a constant $m_i$ such that
	$$\sum_{i}m_i\psi(x_ix)=0$$
for all even $x$ (i.e $v(x)\equiv 0 \mod 2$) with $|x|$ large enough. Then $m_i=0$ for all $i$. 
\end{lemm}
\begin{proof}
	Suppose that $m_{i_0}\neq 0$. At the cost of multiplying by $\psi(-x_{i_0}x)$ we may assume that our relation takes the form:
	$$1=\sum m_i\psi(x_ix),$$
	where $x_i\neq 0$.
	We choose an $b$ with $b$ large and integrate this identity over the set $|x|=|b^2|$ against the multiplicative Haar measure. The left hand side gives a positive value. On other hand, for fixed $x_i\neq 0$ and $|b|$ large, the integral 
	$$\int_{|x|=|b^2|}\psi(x_ix)d^\times x$$ vanishes. Thus the terms of right hand side contribute zero. We get then a contradiction.
\end{proof}

For $x\in\{0,1\}$, we denote by 
$$\Delta_{x,r}(\beta)=\zeta(-1)^{\left[\frac{r+x}{2}\right]\left[\frac{r+x+1}{2}\right]v(\beta)}|\beta|^{\frac{1}{2}\left[\frac{r}{2}\right]-\left[\frac{r^2}{4}\right]}\gamma(\beta,\psi)^{\left[\frac{r+x}{2}\right]}(\beta,\varpi)^{\left[\frac{r+x}{2}\right]}.$$ We have then 
$\Delta_{0,r}(\beta)=\Delta_{w_{\G_r}}(\beta)$ and $\Delta_{1,r}(\beta)=\Delta'_{w_{\G_r}}(\beta)$.

Let $M$ be the standard Levi-subgroup of $\G_r$ of type $(r_1,\dots,r_m)$. Let $\mathbf{t}=\D(t_i\mathrm{Id}_{r_i},1\leq i\leq m)\in T_M$. For $x\in\{0,1\}$, we define a sequence $(y^m_{x,i})_{1\leq i\leq m}\in\{0,1\}^m$ as follows:
\begin{itemize}
	\item $y^m_{x,m}=x$,
	\item $y^m_{x,i}=(y_{x,{i+1}}+r_i)\mod 2\quad\forall 1\leq i\leq m-1$.
\end{itemize} 
We denote by $\Delta_{x,M}$ the function over $T_M(F)$:
$$\Delta_{x,M}(\mathbf{t}):=\left[\zeta(-1)^{\sum_{i=1}^{m-1}r_i\left(\sum_{j=i+1}^{m}\left[\frac{r_j+y^m_{x,j}}{2}\right]\right)v(t_i)}\right]\left[\prod_{i=1}^{m-1}|t_i|^{\frac{-r_i}{2}(\sum_{j=i+1}^{m}r_j)}\right]\prod_{i=1}^{m}\Delta_{y^m_{x,i},r_i}(t_i).$$ 
\begin{theo}\label{main2}
	Let $M$ be the standard Levi-subgroup of $\G_r$ of type $(r_1,\dots,r_m)$. Let $\mathbf{t}=\D(t_i\mathrm{Id}_{r_i},1\leq i\leq m)\in T_M$. We have then:
	$$\Delta_{w_M}(\mathbf{t})=\Delta_{0,M}(\mathbf{t})\quad\text{and}\quad \Delta'_{w_M}(\mathbf{t})=\Delta_{1,M}(\mathbf{t}).$$
\end{theo}
\begin{proof}
	We prove this theorem by induction over $m$.
	\begin{itemize}
		\item If $m=1$, then $M=\G_r$, and $\mathbf{t}=\D(\beta,\dots,\beta)$. We have:
		$$\Delta_{x,\G_r}(\mathbf{t})=\Delta_{x,r}(\beta).$$
		\item Suppose that the theorem holds for $m-1$. We denote by $M_1$ the standard Levi-subgroup of $\G_{r_1+\dots+r_{m-1}}$ of type $(r_1,r_2,\dots,r_{m-1})$. Using Proposition \ref{redu1} for $w=w_M$, $w_1=w_{M_1}$, $w_2=w_{\G_{r_m}}$; $\mathbf{t}_1=\D(t_i\mathrm{Id}_{r_i},1\leq i\leq m-1)$ and $\mathbf{t_2}=t_{m}\mathrm{Id}_{r_m}$, we have
		\begin{eqnarray*}
		\Delta_{w_M}(\mathbf{t})&=&\begin{cases}
			\zeta(-1)^{\frac{r_m}{2}v(\det(\mathbf{t}_1))}|\det(\mathbf{t}_1)|^{-\frac{r_m}{2}}\Delta_{w_1}(\mathbf{t}_1) \Delta_{w_2}(\mathbf{t}_2)&\text{if } r_m\equiv 0\,\mathrm{mod}\, 2\\
			\zeta(-1)^{\frac{r_m-1}{2}v(\det(\mathbf{t}_1))}|\det(\mathbf{t}_1)|^{-\frac{r_m}{2}}\Delta'_{w_1}(\mathbf{t}_1)\Delta_{w_2}(\mathbf{t}_2)&\text{if } r_m\not\equiv 0\,\mathrm{mod}\, 2
		\end{cases}\\
	&=&\zeta(-1)^{\left[\frac{r_m}{2}\right]v(\det(\mathbf{t}_1))}|\det(\mathbf{t}_1)|^{-\frac{r_m}{2}}\Delta_{r_m\, \mathrm{mod}\, 2,M_1}(\mathbf{t_1})\Delta_{0,r_m}(t_m).
		\end{eqnarray*}
	Note that $y^{m-1}_{r_m\, \mathrm{mod}\, 2,i}=y^m_{0,i}$, so the last equation can be rewritten:
	\begin{eqnarray*}
	\Delta_{w_M}(\mathbf{t})&=&\zeta(-1)^{\left[\frac{r_m}{2}\right]v(\det(\mathbf{t}_1))}|\det(\mathbf{t}_1)|^{-\frac{r_m}{2}}\Delta_{r_m\, \mathrm{mod}\, 2,M_1}(\mathbf{t_1})\Delta_{y^m_{0,m},r_m}(t_m)\\
	&=&\zeta(-1)^{\left[\frac{r_m}{2}\right]v(\det(\mathbf{t}_1))}|\det(\mathbf{t}_1)|^{-\frac{r_m}{2}}\Delta_{y^m_{0,m},r_m}(t_m)\\
	&&\times \left[\zeta(-1)^{\sum_{i=1}^{m-2}r_i\left(\sum_{j=i+1}^{m-1}\left[\frac{r_j+y^{m-1}_{r_m\, \mathrm{mod}\, 2,j}}{2}\right]\right)v(t_i)}\right]\\
	&&\times\left[\prod_{i=1}^{m-2}|t_i|^{\frac{-r_i}{2}(\sum_{j=i+1}^{m-1}r_j)}\right]\prod_{i=1}^{m-1}\Delta_{y^{m-1}_{r_m\, \mathrm{mod}\, 2,i},r_i}(t_i)\\
	&=&\zeta(-1)^{\left[\frac{r_m+y^m_{0,m}}{2}\right]\left(\sum_{i=1}^{m-1}(r_iv(t_i))\right)}|\prod_{i=1}^{m-1}t_i^{r_i}|^{-\frac{r_m}{2}}\Delta_{y^m_{0,m},r_m}(t_m)\\
	&&\times \left[\zeta(-1)^{\sum_{i=1}^{m-2}r_i\left(\sum_{j=i+1}^{m-1}\left[\frac{r_j+y^{m}_{0,j}}{2}\right]\right)v(t_i)}\right]\\
	&&\times\left[\prod_{i=1}^{m-2}|t_i|^{\frac{-r_i}{2}(\sum_{j=i+1}^{m-1}r_j)}\right]\prod_{i=1}^{m-1}\Delta_{y^{m}_{0,i},r_i}(t_i)\\
	&=&\Delta_{0,M}(\mathbf{t}).	
	\end{eqnarray*}
We do similarly for $\Delta'_{w_M}(\mathbf{t})$.
	\end{itemize}
\end{proof}
\backmatter

\end{document}